\documentclass[12pt,a4paper]{article}
\usepackage{amssymb,amsmath,amsthm}
\usepackage[english]{babel}
\usepackage{t1enc}
\usepackage[latin2]{inputenc}
\usepackage{epsfig}
\usepackage{subfig}
\usepackage{comment}
\usepackage{epstopdf}
\usepackage{url}
\usepackage{hyperref}
\usepackage{cleveref}
\usepackage{enumitem}
\usepackage[font=small]{caption}

\usepackage[margin=1.1in]{geometry}

\newtheorem{thm}{Theorem}
\newtheorem{manualtheoreminner}{Theorem}
\newenvironment{thm'}[1]{%
	\renewcommand\themanualtheoreminner{#1}%
	\manualtheoreminner
}{\endmanualtheoreminner}

\newtheorem{lem}[thm]{Lemma}

\newtheorem{claim}[thm]{Claim}
\newtheorem{conj}[thm]{Conjecture}

\theoremstyle{remark}

\usepackage{indentfirst}
\mathcode`l="8000
\begingroup
\makeatletter
\lccode`\~=`\l
\DeclareMathSymbol{\lsb@l}{\mathalpha}{letters}{`l}
\lowercase{\gdef~{\ifnum\the\mathgroup=\m@ne \ell \else \lsb@l \fi}}%
\endgroup

\def\cS{\mathcal{S}}

\makeatletter
\def\blfootnote{\gdef\@thefnmark{}\@footnotetext}
\makeatother

\begin{document}

\title{The number of tangencies between two families of curves}
\author{Bal\'azs Keszegh\thanks{Alfréd Rényi Institute of Mathematics and ELTE Eötvös Loránd University, MTA-ELTE Lendület Combinatorial Geometry Research Group, Budapest, Hungary. Research supported by the Lend\"ulet program of the Hungarian Academy of Sciences (MTA), under the grant LP2017-19/2017, by the J\'anos Bolyai Research Scholarship of the Hungarian Academy of Sciences, by the National Research, Development and Innovation Office -- NKFIH under the grant K 132696 and FK 132060 and by the ÚNKP-20-5 New National Excellence Program of the Ministry for Innovation and Technology from the source of the National Research, Development and Innovation Fund.}
\and D\"om\"ot\"or P\'alv\"olgyi\thanks{MTA-ELTE Lend\"ulet Combinatorial Geometry Research Group, Institute of Mathematics, E\"otv\"os Lor\'and University, Budapest, Hungary. Research supported by the Lend\"ulet program of the Hungarian Academy of Sciences (MTA), under the grant LP2017-19/2017.}
}
\maketitle

\begin{abstract}
We prove that the number of tangencies between the members of two families, each of which consists of $n$ pairwise disjoint curves, can be as large as $\Omega(n^{4/3})$. We show that from a conjecture about forbidden $0$-$1$ matrices it would follow that this bound is sharp for doubly-grounded families.
We also show that if the curves are required to be $x$-monotone, then the maximum number of tangencies is $\Theta(n\log n)$, which improves a result by Pach, Suk, and Treml.
Finally, we also improve the best known bound on the number of tangencies between the members of a family of at most $t$-intersecting curves.
\end{abstract}

\section{Introduction}

In this paper, we study the maximum number of tangencies between the members of two families, each of which consists of pairwise disjoint curves.
Pach, Suk, and Treml~\cite{Treml} attributes to Pinchasi and Ben-Dan (personal communication) the first result about this problem, who proved that the maximum number of such tangencies among $n$ curves is $O(n^{3/2} \log n)$.
Their proof is based on a theorem of Marcus and Tardos \cite{marcustardosseq} and of Pinchasi and Radoi\v ci\'c \cite{romrados}; see also the discussion after Claim \ref{claim:p2}.
Moreover, Pinchasi and Ben-Dan suggested that the correct order of magnitude of the maximum may be linear in $n$. 

Pach, Suk, and Treml~\cite{Treml} proved this conjecture in the special case where both families consist of closed convex regions instead of arbitrary curves. 
Ackerman \cite{Ackerman2013} have improved the multiplicative constant from $8$ to $6$ for this case, which is asymptotically optimal.

On the other hand, Pach, Suk, and Treml~\cite{Treml} refuted this conjecture in general by showing an example that the number of tangencies may be more than linear.
More precisely, they constructed two families of $n$ pairwise disjoint $x$-monotone curves with $\Omega(n\log n)$ tangencies between the two families.
They have also shown an upper bound of $O(n\log^2 n)$ for $x$-monotone curves.
Our first result determines the exact order of magnitude for $x$-monotone curves, by giving the following improved upper bound, matching the previous lower bound.

\begin{thm}\label{thm:xmon}
	Given a family of $n$ red and blue $x$-monotone curves such that no two curves of the same color intersect, the number of tangencies between the curves is $O(n\log n)$.
\end{thm}	

Having a closer look at the lower bound construction, one can notice that the two families together form an at most $2$-intersecting family, i.e., any pair of red and blue curves intersects at most twice.
This raised the question if assuming that the family is at most $1$-intersecting leads to a better upper bound.
Ackerman et nos~\cite{ackerman2021tangencies} proved that this is indeed the case, i.e., given an at most $1$-intersecting family of $n$ red and blue curves such that no two curves of the same color intersect, the number of tangencies between the curves is $O(n)$.
Note that we do not need to assume $x$-monotonicity and that it is trivial to construct an example with $\Omega(n)$ tangencies.

\smallskip
For general (not $x$-monotone) curves our main result is a construction which improves the previous $\Omega(n\log n)$ lower bound considerably.
Our lower bound construction has the special property that it contains a red curve and a blue curve that each touch all curves of the other color. Alternately, the construction can be realized such that all curves lie within a vertical strip, every red curve touches the left boundary of the strip and every blue curve touches the right boundary of the strip. We call such a family of curves a \textit{doubly-grounded family}. 

\begin{thm}\label{thm:main}
	There exists a family of $n$ red and blue curves such that no two curves of the same color intersect and the number of tangencies between the curves is $\Omega(n^{4/3})$.
	Moreover, the family can be doubly-grounded.
\end{thm}	

There still remains a polynomial gap between this lower bound and the best known $O(n^{3/2}\log n)$ upper bound. Even getting rid of the $\log n$ from the upper bound would be an interesting improvement. 

To complement this lower bound, we show that for doubly-grounded families this is best possible, provided a conjecture about forbidden $0$-$1$ matrices of Pach and Tardos \cite{Pach2006} holds. The exact statement will be phrased later as Theorem \ref{thm:doubly-grounded}.

The known results about the number of tangencies among two families of pairwise disjoint curves are summarized in Table \ref{table:sum}.

\begin{table}[h]
	\centering
	\begin{tabular}{|c|c|c|c|}
		\hline
		\textbf{Curve type} & \textbf{\# of tangencies} & \textbf{Lower bd ref} & \textbf{Upper bd ref}\\ \hline
		general & $\Omega(n^{4/3})$,  $\tilde O(n^{3/2})$ & Thm.~\ref{thm:main} & Pinchasi \& Ben-Dan \cite{Treml}\\ \hline
		doubly-grounded & Conj.~\ref{conj:01} $\Rightarrow$ $\Theta(n^{4/3})$& Thm.~\ref{thm:main} & Thm.~\ref{thm:doubly-grounded} \\ \hline
		$x$-monotone & $\Theta(n\log n)$ &  Pach et al.~\cite{Treml} & Thm.~\ref{thm:xmon} \\ \hline		
		$\le 1$-intersecting & $\Theta(n)$ & trivial & Ackerman et nos~\cite{ackerman2021tangencies}\\ \hline	
		convex regions & $\Theta(n)$ & trivial & Pach et al.~\cite{Treml}\\ \hline			
	\end{tabular}
	\caption{Summary of results for two families of disjoint curves.}
	\label{table:sum}
\end{table}

\bigskip
A related question is to study the number of tangencies in a single family of curves.
If the curves in the family are not allowed to cross, i.e., they are disjoint apart from tangencies, then it follows from Euler's formula that the number of tangencies is at most $3n-6$ among $n\ge 3$ curves \cite{erdosgrunbaum}.
While settling a conjecture of Richter and Thomassen \cite{rt}, it was shown by Pach, Rubin and Tardos \cite{prt1,prt2} that if the number of tangencies is superlinear in the number of curves, then the number of crossings is superlinear in the number of touchings.

According to a conjecture of Pach \cite{pachpc} the number of tangencies between an at most $1$-intersecting family of $n$ curves is $O(n)$ if every pair of curves intersects.
The conjecture is known to hold for pseudo-circles \cite{ANPPSS},
and in the special case when there exist constantly many faces of the arrangement of the curves such that every curve has one of its endpoints inside one of these faces \cite{GYORGYI201829}. 

The number of tangencies within an at most $1$-intersecting family of $n$ curves can be $\Omega(n^{4/3})$, even if all curves are required to be $x$-monotone; this follows from a famous construction of Erd\H{o}s and Purdy \cite{erdoscombgeo} of $n$ points and $n$ lines that determine $\Omega(n^{4/3})$ point-line incidences by replacing each point with a small curve, and slightly perturbing the lines; see \cite{pachbook}.
For $x$-monotone curves, it is mentioned in \cite{ackerman2021tangencies} that an almost matching upper bound of $O(n^{4/3}\log^{2/3}n)$ follows from a result of Pach and Sharir~\cite{PS91} which can be improved with a more careful analysis to $O(n^{4/3}\log^{1/3}n)$.\footnote{Personal communication with Eyal Ackerman, who observed this with Rom Pinchasi.}
For further results, see \cite{maya,esz,erdosgrunbaum,salazar}.
In particular, in \cite{maya} it was shown that any family of $n$ at most $t$-intersecting curves determines at most $O(n^{2-\frac1{3t+15}})$ tangencies.
We improve this bound as follows.

\begin{thm}\label{thm:onefamily}
	Any family of $n$ at most $t$-intersecting curves determines at most $O(n^{2-\frac1{t+3}})$ tangencies.
\end{thm}

For a family of curves its \emph{tangency graph} is defined as the graph which has a vertex for each curve and two vertices are connected if and only if the corresponding curves are tangent to each other. We actually prove that for some large enough $c=c(t)$ the tangency graph of any family of at most $t$-intersecting curves avoids $K_{t+3,c}$ as a subgraph. Theorem \ref{thm:onefamily} follows from this directly using the K\H ov\'ari--S\'os--Tur\'an theorem. 

\subsection{Technical definitions}
All curves considered in this paper are \emph{simple} and planar, i.e., each is the image of an injective continuous function from an (open or closed) interval into the plane, also known as a Jordan arc. 

For simplicity, we assume that any two curves intersect in a finite number of points, and that no three curves intersect in the same point.
An intersection point $p$ of two curves is a \emph{crossing} point if there is a Jordan region $D$ which contains $p$ but no other intersection point of these two curves in its interior, each curve intersects the boundary of $D$ at exactly two points, and in the cyclic order of these four points no two consecutive points belong to the same curve. 
We say that two curves \emph{touch} each other at a \emph{tangency} point $p$ if both of them contain $p$ in their relative interior, $p$ is their only intersection point, and $p$ is not a crossing point.
A family of curves is \emph{at most $t$-intersecting} if every pair of curves intersects in at most $t$ points.
Note that without the $t$-intersecting condition (which in this paper we do not require when studying tangencies among two families), when counting tangencies, it makes no difference whether we consider curves or Jordan regions, so our results also hold when each of the two families consist of disjoint Jordan regions.


\section{Upper bound for a single family of curves}

\begin{lem}\label{lem:forb}
	For every positive integer $t$ there exists a constant $c=c(t)$ such that given a family of at most $t$-intersecting curves, the tangency graph of the curves cannot contain $K_{t+3,c}$ as a subgraph.
\end{lem}

\begin{proof}
	Suppose for a contradiction that there is family of curves such its tangency graph contains $K_{t+3,c}$ as a subgraph. Call the $t+3$ curves that form one part of this bipartite graph \emph{red}, and the $c$ curves that form the other part of the bipartite graph \emph{blue}. Each red curve is cut by the other $t+2$ red curves into at most $1+(t+2)t$ parts, and by the pigeonhole principle there are at least $c_2=c/(1+(t+2)t)^{t+3}$ blue curves that intersect the same part from each red curve. Choosing these red parts, from now on we assume that we have a family of red and blue curves whose tangency graph contains a $K_{t+3,c_2}$ and the red curves are pairwise disjoint.
	Moreover, at least $c_3=c_2/(t+3)!$ blue curves touch these $t+3$ red curves in the same order, where we consider each curve as a continuous image of $[0,1]$.
	
	Denote the red curves by $\gamma_1,\ldots,\gamma_{t+3}$ in the order in which they are touched by the two blue curves. 	
	At least $c_4=c_3/2^{t+1}$ blue curves turn in the same direction at each of their $t+1$ tangency points with the red curves $\gamma_2,\ldots,\gamma_{t+2}$, i.e., all red curves that are not the first or last that they touch.
	We choose $c$ such that $c_4\ge 2$.
	
	Going from $\gamma_1$ to $\gamma_2$, and then from $\gamma_2$ to $\gamma_3$, at $\gamma_2$ both blue curves either turn left, or right. See Figure \ref{fig:gamma123}.
	Then either those parts of these two blue curves intersect that go from $\gamma_1$ to $\gamma_2$, or the $\gamma_1$ to $\gamma_2$ part of one curve intersects the $\gamma_2$ to $\gamma_3$ part of the other curve.
	In either case, we can delete $\gamma_1$, and proceed by induction to get $t+1$ intersections between the two blue curves, contradicting our assumption.
\end{proof}

\begin{figure}[h]
	\centering
	\includegraphics[width=12cm]{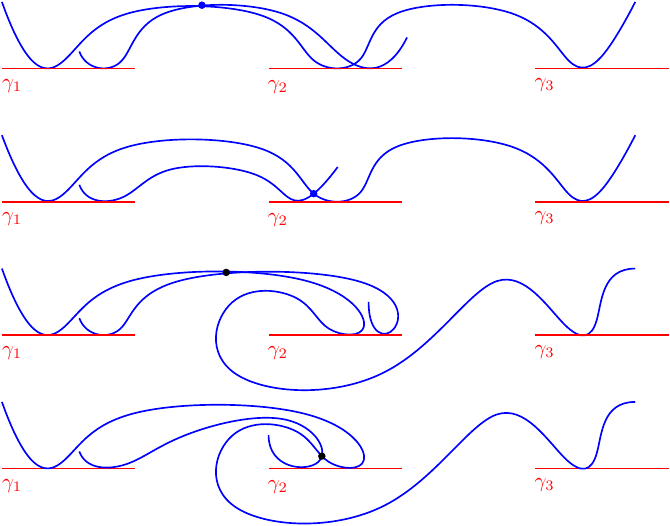}
	\caption{The two possible turns at $\gamma_2$ and the respective two possible intersections among the two blue curves. Note that it does not matter in which order or from which side they touch the red segments as those could also be replaced by red disks without changing the combinatorial layout of the tangencies.}
	\label{fig:gamma123}
\end{figure}

\begin{proof}[Proof of Theorem \ref{thm:onefamily}.]
	The K\H ov\'ari--S\'os--Tur\'an theorem \cite{kst} concerning the well-known Zarankiewicz problem states that if a bipartite graph $G$ on $n+n$ vertices does not contain $K_{a,b}$ as a subgraph, 
	then it has $O(n^{2-1/a})$ edges. By Lemma \ref{lem:forb} the tangency graph does not contain $K_{a,b}$ as a subgraph for $a=t+3$ and $b=c$ and thus this theorem implies that it has	$O(n^{2-\frac1{t+3}})$ edges.
\end{proof}

\section{Upper bound for the $x$-monotone case}

For the proof, we will use the following result from \cite{edgeordered}.

\begin{thm}[Gerbner et al.~\cite{edgeordered}]\label{thm:edgeordered}
	A graph with a total ordering $<$ on its edges such that it does not contain as a subgraph a path on $5$ vertices, $\{a,b,c,d,e\}$, whose edges are ordered as $ab<cd<bc<de$, can have at most $O(n\log n)$ edges.
\end{thm}

\begin{proof}[Proof of Theorem \ref{thm:xmon}]
Let $\cS$ be the family of $n$ red and blue $x$-monotone curves. After a slight perturbation we can assume that no two tangencies (as points in the plane) have the same $x$-coordinate. Let $G$ be the graph whose vertices correspond to the curves and whose edges correspond to the tangent pairs. Let $G_1$ (resp. $G_2$) be the subgraph where we take an edge only if in a small neighborhood of the corresponding tangency the red curve is above (resp. below) the blue curve. Note that $G$ is the edge-disjoint union of $G_1$ and $G_2$. We will show that the edges of $G_1$ can be ordered in such a way that it will satisfy the conditions of Theorem \ref{thm:edgeordered}, thus $G_1$ has $O(n\log n)$ edges. Since the same holds for $G_2$ due to symmetry, in turn this implies the same for $G$ itself, proving the theorem.

\begin{figure}[h]
	\centering
	\includegraphics[width=\textwidth]{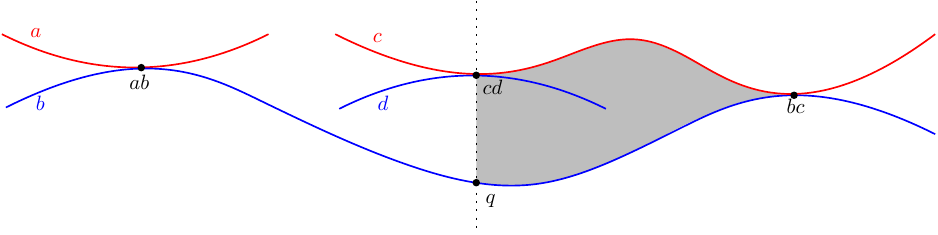}
	\caption{It is impossible to realize a certain path on $5$ vertices in $G_1$.}
	\label{fig:xmon}
\end{figure}

Order the edges of $G_1$ according to the $x$-coordinate of the corresponding tangencies. This gives a complete ordering of the edges. We need to show that $G_1$ cannot contain a path on $5$ vertices, $\{a,b,c,d,e\}$, such that $ab<cd<bc<de$ in the order of the edges (see Figure \ref{fig:xmon}).  Indeed, assume on the contrary. Without loss of generality, we can assume that $a$ corresponds to a red curve. Note that as $ab$ (i.e., the corresponding tangency) is left to $cd$ and $cd$ is left to $bc$, the curve $b$ must intersect the vertical line passing through $cd$ in a point $q$. Moreover, as $b$ touches $c$ from below (in $bc$), $q$ must be also below $cd$. This means that the part of the curve $d$ that lies to the right from $cd$ must be in the closed region determined by the part of $c$ between $cd$ and $bc$, the part of $b$ between $q$ and $bc$, and the vertical segment between $cd$ and $q$. This implies that $d$ has no point to the right from $bc$, contradicting that $bc<de$. 
\end{proof}

\section{Lower bound for the general case}

\begin{proof}[Proof of Theorem \ref{thm:main}]

Take a construction with $n$ lines, $l_1,\ldots,l_n$, and $n$ points, $p_1,\ldots,p_n$, that has $\Omega(n^{4/3})$ incidences, using a famous construction of Erd\H{o}s and Purdy \cite{erdoscombgeo}, which is sharp, as shown by the Szemer\'edi--Trotter theorem \cite{SzT1983}.
We can suppose that none of the $n$ lines is vertical or horizontal, and no two points form a vertical line.
Take a large axis-parallel box $B$ such that all crossings lie within $B$ and all lines intersect the left and right sides of $B$.
We will convert the lines into red curves, and the points into blue curves, such that each incidence will become a tangency.
Each point $p_j$ is replaced by a small blue circle $D_j$ whose bottommost point is $p_j$ and a segment going upwards from the top of the circle to the top side of $B$ ( we can delete a small arc of the circle adjacent to its top so that together with the segment they form a simple curve, as required).

\begin{figure}[h]
	\centering
	\includegraphics[width=\textwidth]{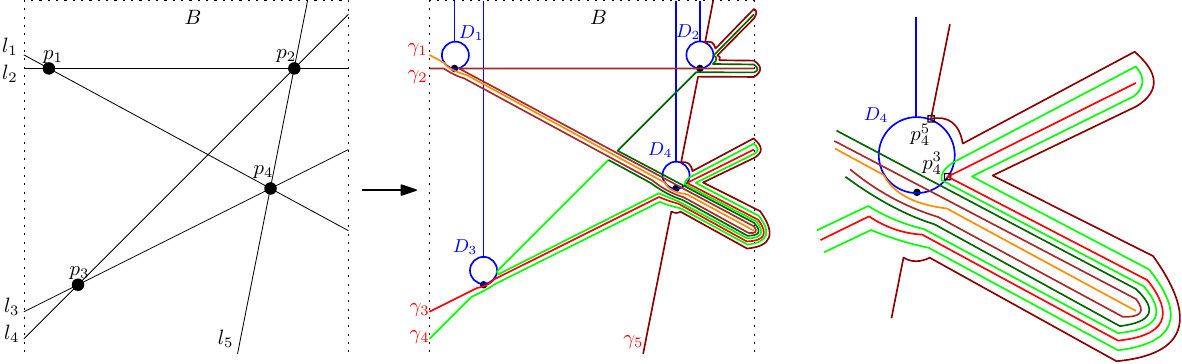}
	\caption{The lower bound construction, `red' curves are drawn with various colors and $\gamma_4$ is drawn with two shades of green to ease readability. The right side figure is an enlargement of the part of the redrawing in the vicinity of $D_4$.}
	\label{fig:sztconstr}
\end{figure}

We will place the red curves, $\gamma_1,\ldots,\gamma_n$, one-by-one in the order given by the slopes of the corresponding lines, i.e., the order in which they intersect the left side of $B$, starting with the topmost one (see Figure \ref{fig:sztconstr}).
The two ends of $\gamma_i$ will be the intersection points of the sides of $B$ with $l_i$.
We will also maintain the following.
For a given line $l_i$ the corresponding red curve $\gamma_i$ lies in the closed halfplane whose boundary is $l_i$ and lies below $l_i$. 
The curves never cross the left, top or bottom side of $B$ but they may go slightly out at the right side. The curve $\gamma_i$ has parts that lie on $l_i$. Whenever $\gamma_i$ leaves $l_i$ (when going from left to right) at some point $x\in l_i$, it later returns to $l_i$ at a point $x'$ very close to and a bit to the right from $x$. Between $x$ and $x'$, all points of $\gamma_i$ lie below $l_i$ and have an $x$-coordinate at least as big as the $x$-coordinate of $x$. The distance between $x$ and $x'$ can be chosen to be at most the diameter of the small blue circles.

In general, when we draw the next curve $\gamma_i$, it starts at the intersection point of $l_i$ and the left side of $B$, and starts to follow $l_i$ to the right.\\

Whenever $l_i$ is incident to some point $p_j$, it intersects the disk $D_j$ in two points, one of which is $p_j$; denote the other intersection point by $p_j^i$.
Since the lines were ordered by their slopes, the points $p_j^i$ (which exist) follow each other in a counterclockwise order as $i$ increases.
We draw the curve $\gamma_i$ so that it touches $D_j$ at $p_j^i$.\\
If the slope of the line $l_i$ is negative, then after the touching point $p_j^i$ the curve $\gamma_i$ leaves $l_i$ and goes slightly outside of the disk along the $p_j^ip_j$ arc.
It returns to $l_i$ slightly after $p_j$, and continues along $l_i$.\\
If the slope of the line is positive, then the curve $\gamma_i$ leaves $l_i$ slightly before $p_j$ and goes slightly outside of the disk along the $p_jp_j^i$ arc, then returns to $l_i$ at $p_j^i$, and continues along $l_i$.
However, if there is some $h<i$ for which $l_h$ is incident to $p_i$ and also has a positive slope, then instead of going alongside the arc, $\gamma_i$ might be forced to make a much larger detour; the description of this detour will be detailed in the next paragraph.
The only important part for us here is that whenever the first time $\gamma_i$ returns from the detour so that it is not locally separated from $D_j$ anymore, it touches $D_j$ at $p_j^i$, and continues along $l_i$.\\

Whenever $\gamma_i$ would intersect another red curve $\gamma_h$, $h<i$, at some point $x$, instead of creating an intersection, using that we ordered the lines by their slopes, we follow $\gamma_h$, going very close to it, until we reach the endpoint of $\gamma_h$ on the right side of $B$. Then we go back on the other side of $\gamma_h$, again very close to it, until we arrive back at some point $x'\in \gamma_i\cap l_i$ in a close vicinity of $x$.

Notice that if a line $l_i$ goes through some $p_j$, then $\gamma_i$ will go in the above defined way one-by-one around every curve drawn earlier that touches $D_j$ and only then will it touch $D_j$ at $p_j^i$.

In case of going around some $\gamma_h$, the part of $\gamma_h$ which $\gamma_i$ went around lies below $l_h$ and to the right of $l_i\cap l_h$, thus also below $l_i$.
This maintains the condition for $\gamma_i$ as well, and it 
 also implies that while going around some $\gamma_h$, $\gamma_i$ can never intersect any disk $D_j$ that it is supposed to touch.

To see that this construction can be made doubly-grounded, we can easily deform it to lie in a vertical strip.
We can transform the left side of $B$ into a segment on the left side of the strip, and the top side of $B$ into a segment on the right side of the strip.
\end{proof}

\section{Conditional upper bound for the doubly-grounded case}

We first define a certain redrawing of the curves to make some arguments easier (see Figure \ref{fig:redrawing}).
For each curve $\gamma$ we choose an arbitrary point on it, $p_\gamma$. Then, we draw a star-shaped curve $\gamma^*$ 
such that $p_\gamma\in \gamma^*$ will be a special point in $\gamma^*$, called the center, and for each tangency $x$ of $\gamma$ and some other curve of the family, we draw an arc 
from $p_\gamma$ to $x$ that goes very close to the original curve. We repeat this for every curve transforming each into such a star, while preserving disjointnesses and tangencies. 

\begin{figure}[h]
	\centering
	\includegraphics[width=\textwidth]{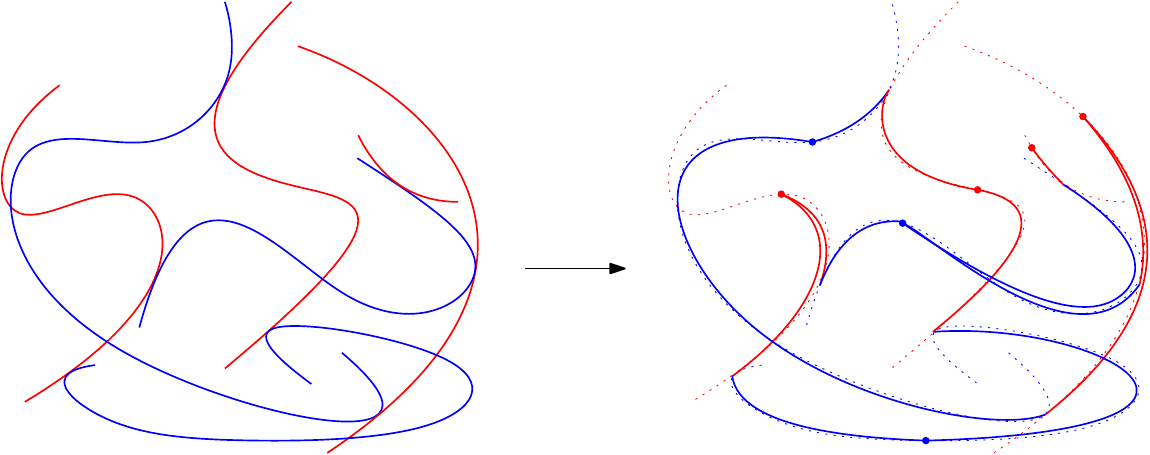}
	\caption{Drawing of a graph based on a set of curves.}
	\label{fig:redrawing}
\end{figure}

This drawing can be regarded as a drawing $D$ of a bipartite graph $G$ in the plane, whose vertices correspond to the original curves, represented by the centers of the stars, and whose edges correspond to the tangencies.
Note that each edge consists of one red and one blue part; we refer to these as \emph{edge-parts}.
Thus, the number of edges of $G$ is exactly the number of tangencies in the original setting.
Also, if two vertices are adjacent in $G$, then none of their neighboring edge-parts are allowed to intersect.
Indeed, if $uv$ is an edge, then the star-shaped curves that correspond to $u$ and $v$ can only meet at their point of tangency.
Nor are two edge-parts of the same color allowed to intersect.
These observations immediately imply the following.


\begin{claim}\label{claim:p2}
	The drawing $D$ of $G$ obtained by the above redrawing process cannot contain a self-crossing path on five edge-parts.
	Therefore, it can contain neither a self-crossing $C_4$, nor 
	a self-crossing path on two edges. 
\end{claim}

\begin{proof}
	Without loss of generality, suppose that the five edge-parts of a path are red, red, blue, blue, red.	
	Denote the vertex between the first two red parts by $u$, the vertex between the two blue parts by $v$, and the vertex incident to the last red part by $w$.
	Recall that edge-parts of the same color do not intersect.
	Both blue edge-parts are incident to $v$, so they cannot intersect the edge-parts incident to a vertex adjacent to $v$.
	But each red edge-part is incident to either $u$ or $w$, which are both adjacent to $v$.
	Therefore, the path cannot be self-crossing.
\end{proof}


%
%

In \cite{marcustardosseq} it was proved that a graph $G$ which can be drawn without a self-crossing $C_4$ has at most $O(n^{3/2}\log n)$ edges. As the number of edges of $G$ equals the number of tangencies between curves, this implies the upper bound of $O(n^{3/2}\log n)$ on the number of tangencies between the curves that was mentioned in the introduction. Now we use this representation to prove our better upper bound in case the curves are doubly-grounded.

\begin{figure}[h]
	\centering
	\includegraphics[width=\textwidth]{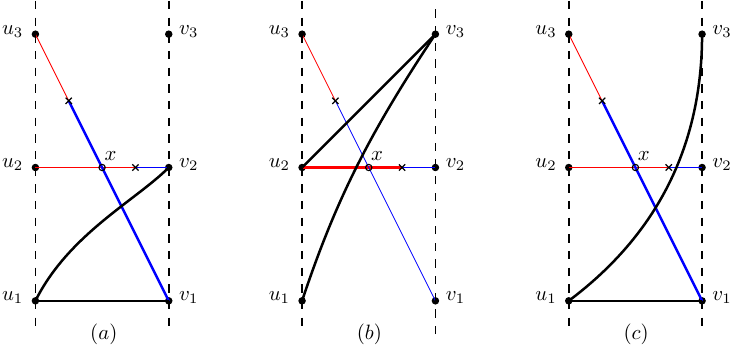}
	\caption{In every case there exists a crossing path on 5 edge-parts.}
	\label{fig:doubly}
\end{figure}

Our main result of this section shows an upper bound matching the lower bound for doubly-grounded curves, assuming a conjecture about forbidden $0$-$1$ matrices. For more details about the following definitions see, e.g., \cite{Pach2006}. Briefly, a $0$-$1$ matrix $M$ is said to contain another $0$-$1$ matrix $P$ if there is a subset of rows and columns of $M$ such that in the intersection of these rows and columns there is a matrix $M'$ such that we can get $P$ from $M'$ by possibly replacing some $1$-entries by $0$-entries.\footnote{Note that the order of the rows and columns is fixed in both $M$ and $P$.} A $0$-$1$ matrix is a cycle if it is the adjacency matrix of a cycle graph. We omit the definition of positive orthogonal cycles in general, we just say that positive orthogonal $6$-cycles are those $3\times 3$ matrices that are $6$-cycles and their middle entry is a $1$-entry.

Pach and Tardos \cite{Pach2006} proved the following.

\begin{thm}\label{thm:01}
	If a $0$-$1$ matrix avoids all positive orthogonal cycles, then it has $O(n^{4/3})$ $1$-entries.
\end{thm}

In fact, it may be enough to forbid only positive orthogonal $6$-cycles to achieve the same bound, as mentioned in \cite{Pach2006}.

\begin{conj}\label{conj:01}
	If a $0$-$1$ matrix avoids all positive orthogonal $6$-cycles, then it has $O(n^{4/3})$ $1$-entries.
\end{conj}

Actually, it may be enough to forbid only one $6$-cycle and it may not even be necessary that it is a positive orthogonal cycle, as conjectured in \cite{gyoric6} and \cite{Methuku2022}.\footnote{In \cite{gyoric6} they show for various subsets of the $6$-cycles that forbidding them implies the $O(n^{4/3})$ upper bound on the edges, but unfortunately all of these subsets contain non-positive cycles as well, so they do not imply Conjecture \ref{conj:01}.}

\begin{conj}\label{conj:01strong}
	If a $0$-$1$ matrix avoids a given $6$-cycle, then it has $O(n^{4/3})$ $1$-entries.
\end{conj}

We show that the weaker one of these two conjectures would already imply that our lower bound for doubly-grounded curves is optimal.

\begin{thm}\label{thm:doubly-grounded}
	Given a family of $n$ red and blue curves that lie within a vertical strip such that no two curves of the same color intersect, every red curve touches the left boundary of the strip and every blue curve touches the right boundary of the strip, then the number of tangencies between the curves is $O(n^{4/3})$ if Conjecture \ref{conj:01} holds.
\end{thm}

\begin{proof}
	Given a doubly-grounded family of curves, for each curve $\gamma$ choose the point where it touches the appropriate boundary of the vertical strip to be $p_\gamma$, and do the redrawing defined above. We get a bipartite graph $G$ on $n$ vertices whose vertices lie on the two boundaries of the strip and whose edges are drawn inside the strip. 
	
	We have seen that this graph $G$ has no self-crossing $C_4$.\footnote{In case of a doubly-grounded family, this is trivial.}
	As in a doubly-grounded representation any $C_4$ is self-crossing, this already implies that $G$ is $C_4$-free, thus has at most $O(n^{3/2})$ edges.	
	Now we prove that additionally it does not contain a certain type of $C_6$ that we call \emph{positive}, following \cite{Pach2006}.
	A $C_6$ on vertices $u_1,u_2,u_3,v_1,v_2,v_3$ is positive if $u_1,u_2,u_3$ are on the left side of the strip in this order, $v_1,v_2,v_3$ are on the right side in this order, and $u_2v_2$ is an edge of the $C_6$.\footnote{In fact, it would be sufficient to assume that $u_2v_2$ is an edge of $G$, as since there is no self-crossing $C_4$, this can only happen if $u_2v_2$ is an edge of the $C_6$.}
	
	Assume on the contrary that such a positive $C_6$ exists. 
	Recall that each edge $u_iv_j$ has a red part incident to $u_i$ and a blue part incident to $v_j$, and that no parts of the same color can cross.
	
	Without loss of generality, we can assume that $u_3v_1$ is an edge of the $C_6$.
	The edges $u_2v_2$ and $u_3v_1$ must cross, denote (one of) their crossing point(s) by $x$.
	Without loss of generality, we can assume that the red part of $u_2v_2$ crosses the blue part of $u_3v_1$.
	This implies that $u_2v_1$ cannot be an edge, as otherwise $xv_1u_2x$ would be a self-crossing path on 4 edge-parts contradicting Claim \ref{claim:p2}.
	So $u_2v_3$ and $u_1v_1$ need to be edges. 
	
	Then $u_2x$ (contained in the red part of $u_2v_2$) and $xv_1$ (contained in the blue part of $u_3v_1$) separate $u_1$ from $v_2$ and from $v_3$ (within the strip).
	However, at least one of $u_1v_2$ and $u_1v_3$ must be an edge of $G$.\\
	If $u_1v_2$ is an edge, then as it cannot intersect $u_2v_2$, it must intersect $xv_1$, so $v_2u_1v_1x$ is a self-crossing path on 5 edge-parts, see Figure \ref{fig:doubly}(a).\\
	If $u_1v_3$ is an edge, then either $u_1v_3$ intersects $u_2x$ and then $u_1v_3u_2x$ is a self-crossing path on 5 edge-parts, see Figure \ref{fig:doubly}(b); or $u_1v_3$ intersects $xv_1$ and then $v_3u_1v_1x$ is a self-crossing path on 5 edge-parts, see Figure \ref{fig:doubly}(c).\\
	All cases contradict Claim \ref{claim:p2}.

	
	Thus, we have shown that positive $C_6$'s are not in $G$. Take the adjacency $0$-$1$ matrix of the bipartite graph $G$ such that the rows and columns are ordered according to the order of the corresponding points along the boundaries of the strip. Recall that the positive $6$-cycles are the ones in whose $3\times 3$ matrix representation the middle entry is a $1$-entry. Thus, Conjecture \ref{conj:01} implies that $G$ has at most these many edges, which in turn implies the required upper bound on the number of tangencies in the original setting of curves.	
\end{proof}

Theorem \ref{thm:doubly-grounded} implies the following slight extension in a standard way.

\begin{thm}\label{thm:circle-grounded}
	Given a family of $n$ red and blue curves that lie inside a circle such that no two curves of the same color intersect and every curve touches the circle once, then the number of tangencies between the curves is $O(n^{4/3})$ if Conjecture \ref{conj:01} holds.
\end{thm}

\begin{proof}
	We prove by induction that the number of touchings is at most  $dn ^{4/3}$ for a suitable $d$. This certainly holds for small $n$. For a general $n$, by a continuity argument we can split the boundary circle into two arcs such that each arc touches at most $\lceil n/2\rceil$ curves of each color. Then the number of touchings between the red curves (resp.\ blue curves) touching the first arc and the blue curves (resp.\ red curves) touching the second arc is at most $cn^{4/3}$ by Theorem \ref{thm:doubly-grounded}, where the $c$ is hidden in the $O$ notation. The touchings between the red and blue curves touching the same arc is at most $d\lceil n/2\rceil ^{4/3}$ by induction (after deforming the drawing such that the arc becomes a circle). 
	Thus, the total number of touchings is at most $2d\lceil n/2\rceil ^{4/3}+2cn^{4/3}\le dn^{4/3}$ for $d$ large enough (depending on $c$), finishing the proof.
\end{proof}

We have seen that both in the general case (see the paragraph after the proof of Claim \ref{claim:p2}) and in the doubly-grounded case (see Theorem \ref{thm:doubly-grounded}) the best known upper bounds use only Claim \ref{claim:p2}. Besides improving the lower and upper bounds in the general case, it would be interesting to decide if it is possible to improve the general upper bound using only Claim \ref{claim:p2}. Also, this connection makes it even more interesting to solve Conjectures \ref{conj:01} and \ref{conj:01strong}. Note that any upper bound for the respective problems that is better than $\tilde O(n^{3/2})$ would also improve our bound on the doubly-grounded case.

\smallskip
\noindent{\bf Acknowledgment}
We thank Eyal Ackerman and J\'anos Pach for discussions about these problems, especially on Theorem \ref{thm:onefamily}.

\footnotesize
\bibliographystyle{plainurl}
\bibliography{psdisk}

\begin{thebibliography}{10}

\bibitem{Ackerman2013}
Eyal Ackerman.
\newblock {\em The Maximum Number of Tangencies Among Convex Regions with a
  Triangle-Free Intersection Graph}, pages 19--30.
\newblock Springer New York, New York, NY, 2013.
\newblock \href {https://doi.org/10.1007/978-1-4614-0110-0_3}
  {\path{doi:10.1007/978-1-4614-0110-0_3}}.

\bibitem{ackerman2021tangencies}
Eyal Ackerman, Bal{\'{a}}zs Keszegh, and D{\"{o}}m{\"{o}}t{\"{o}}r
  P{\'{a}}lv{\"{o}}lgyi.
\newblock On tangencies among planar curves with an application to coloring
  {L}-shapes.
\newblock In Jaroslav Ne{\v{s}}et{\v{r}}il, Guillem Perarnau, Juanjo Ru{\'e},
  and Oriol Serra, editors, {\em Extended Abstracts EuroComb 2021}, pages
  123--128, Cham, 2021. Springer International Publishing.
\newblock \href {https://doi.org/10.1007/978-3-030-83823-2_20}
  {\path{doi:10.1007/978-3-030-83823-2_20}}.

\bibitem{ANPPSS}
Pankaj~K. {Agarwal}, Eran {Nevo}, J\'anos {Pach}, Rom {Pinchasi}, Micha
  {Sharir}, and Shakhar {Smorodinsky}.
\newblock {Lenses in arrangements of pseudo-circles and their applications}.
\newblock {\em {J. ACM}}, 51(2):139--186, 2004.
\newblock \href {https://doi.org/10.1145/972639.972641}
  {\path{doi:10.1145/972639.972641}}.

\bibitem{maya}
Maya Bechler{-}Speicher.
\newblock A crossing lemma for families of jordan curves with a bounded
  intersection number.
\newblock {\em CoRR}, abs/1911.07287, 2019.
\newblock URL: \url{http://arxiv.org/abs/1911.07287}, \href
  {http://arxiv.org/abs/1911.07287} {\path{arXiv:1911.07287}}.

\bibitem{esz}
Jordan~S. {Ellenberg}, Jozsef {Solymosi}, and Joshua {Zahl}.
\newblock {New bounds on curve tangencies and orthogonalities}.
\newblock {\em {Discrete Anal.}}, 2016:22, 2016.
\newblock Id/No 18.
\newblock \href {https://doi.org/10.19086/da990} {\path{doi:10.19086/da990}}.

\bibitem{erdoscombgeo}
Paul {Erd\H{o}s}.
\newblock {Problems and results in combinatorial geometry}.
\newblock {Discrete geometry and convexity, Proc. Conf., New York 1982, Ann.
  N.Y. Acad. Sci. 440, 1--11}, 1985.

\bibitem{erdosgrunbaum}
Paul {Erd\H{o}s} and Branko {Gr\"unbaum}.
\newblock {Osculation vertices in arrangements of curves}.
\newblock {\em {Geom. Dedicata}}, 1:322--333, 1973.
\newblock \href {https://doi.org/10.1007/BF00147765}
  {\path{doi:10.1007/BF00147765}}.

\bibitem{gyoric6}
{Ervin Gy\H ori and D\'aniel Kor\'andi and Abhishek Methuku and Istv\'an Tomon
  and Casey Tompkins and M\'at\'e Vizer}.
\newblock {On the Tur{\'a}n number of some ordered even cycles}.
\newblock {\em European Journal of Combinatorics}, 73:81--88, 2018.
\newblock URL:
  \url{https://www.sciencedirect.com/science/article/pii/S0195669818300982},
  \href {https://doi.org/https://doi.org/10.1016/j.ejc.2018.05.008}
  {\path{doi:https://doi.org/10.1016/j.ejc.2018.05.008}}.

\bibitem{edgeordered}
D\'aniel Gerbner, Abhishek Methuku, D\'aniel~T. Nagy, D{\"{o}}m{\"{o}}t{\"{o}}r
  P{\'{a}}lv{\"{o}}lgyi, G\'abor Tardos, and M\'at\'e Vizer.
\newblock {Tur\'an problems for Edge-ordered graphs}, 2021.
\newblock \href {http://arxiv.org/abs/2001.00849} {\path{arXiv:2001.00849}}.

\bibitem{GYORGYI201829}
P\'eter Gy{\"o}rgyi, B\'alint Hujter, and S\'andor Kisfaludi-Bak.
\newblock On the number of touching pairs in a set of planar curves.
\newblock {\em Computational Geometry}, 67:29--37, 2018.
\newblock \href {https://doi.org/https://doi.org/10.1016/j.comgeo.2017.10.004}
  {\path{doi:https://doi.org/10.1016/j.comgeo.2017.10.004}}.

\bibitem{kst}
Tam{\'a}s K{\H o}v{\'a}ri, Vera~T. S{\'o}s, and Paul Tur{\'a}n.
\newblock {On a problem of K. Zarankiewicz}.
\newblock {\em Colloquium Mathematicum}, 3:50--57, 1954.

\bibitem{marcustardosseq}
Adam Marcus and G\'abor Tardos.
\newblock Intersection reverse sequences and geometric applications.
\newblock {\em Journal of Combinatorial Theory, Series A}, 113(4):675--691,
  2006.
\newblock URL:
  \url{https://www.sciencedirect.com/science/article/pii/S0097316505001251}.

\bibitem{Methuku2022}
Abhishek Methuku and Istv{\'a}n Tomon.
\newblock Bipartite tur{\'a}n problems for ordered graphs abhishek methuku,
  istv{\'a}n tomon.
\newblock {\em Combinatorica}, Jan 2022.

\bibitem{pachpc}
J\'anos Pach.
\newblock personal communication.

\bibitem{pachbook}
J\'anos Pach and Pankaj~K. Agarwal.
\newblock {\em Combinatorial Geometry}, chapter~11, pages 167--181.
\newblock John Wiley and Sons Ltd, 1995.
\newblock \href {https://doi.org/https://doi.org/10.1002/9781118033203.ch11}
  {\path{doi:https://doi.org/10.1002/9781118033203.ch11}}.

\bibitem{prt1}
J\'anos Pach, Natan Rubin, and G\'abor Tardos.
\newblock On the {R}ichter-{T}homassen conjecture about pairwise intersecting
  closed curves.
\newblock {\em Combinatorics, Probability and Computing}, 25(6):941--958, 2016.
\newblock \href {https://doi.org/10.1017/S0963548316000043}
  {\path{doi:10.1017/S0963548316000043}}.

\bibitem{prt2}
J\'anos Pach, Natan Rubin, and G\'abor Tardos.
\newblock A crossing lemma for {J}ordan curves.
\newblock {\em Advances in Mathematics}, 331:908--940, 2018.
\newblock \href {https://doi.org/https://doi.org/10.1016/j.aim.2018.03.015}
  {\path{doi:https://doi.org/10.1016/j.aim.2018.03.015}}.

\bibitem{PS91}
J{\'{a}}nos Pach and Micha Sharir.
\newblock On vertical visibility in arrangements of segments and the queue size
  in the {B}entley-{O}ttmann line sweeping algorithm.
\newblock {\em {SIAM} J. Comput.}, 20(3):460--470, 1991.
\newblock \href {https://doi.org/10.1137/0220029} {\path{doi:10.1137/0220029}}.

\bibitem{Treml}
J{\'{a}}nos Pach, Andrew Suk, and Miroslav Treml.
\newblock Tangencies between families of disjoint regions in the plane.
\newblock {\em Comput. Geom.}, 45(3):131--138, 2012.
\newblock \href {https://doi.org/10.1016/j.comgeo.2011.10.002}
  {\path{doi:10.1016/j.comgeo.2011.10.002}}.

\bibitem{Pach2006}
J{\'a}nos Pach and G{\'a}bor Tardos.
\newblock Forbidden paths and cycles in ordered graphs and matrices.
\newblock {\em Israel Journal of Mathematics}, 155(1):359--380, Dec 2006.
\newblock \href {https://doi.org/10.1007/BF02773960}
  {\path{doi:10.1007/BF02773960}}.

\bibitem{romrados}
Rom Pinchasi and Rado{\v s} Radoi{\v c}i\'c.
\newblock Topological graphs with no self-intersecting cycle of length 4.
\newblock In {\em Proceedings of the Nineteenth Annual Symposium on
  Computational Geometry}, SCG '03, pages 98--103, New York, NY, USA, 2003.
  Association for Computing Machinery.
\newblock \href {https://doi.org/10.1145/777792.777807}
  {\path{doi:10.1145/777792.777807}}.

\bibitem{rt}
R.~B. {Richter} and C.~{Thomassen}.
\newblock {Intersections of curve systems and the crossing number of \(C_
  5\times C_ 5\)}.
\newblock {\em {Discrete Comput. Geom.}}, 13(2):149--159, 1995.
\newblock \href {https://doi.org/10.1007/BF02574034}
  {\path{doi:10.1007/BF02574034}}.

\bibitem{salazar}
Gelasio {Salazar}.
\newblock {On the intersections of systems of curves}.
\newblock {\em {J. Comb. Theory, Ser. B}}, 75(1):56--60, 1999.
\newblock \href {https://doi.org/10.1006/jctb.1998.1858}
  {\path{doi:10.1006/jctb.1998.1858}}.

\bibitem{SzT1983}
E.~Szemer{\'e}di and W.~T. Trotter.
\newblock Extremal problems in discrete geometry.
\newblock {\em Combinatorica}, 3(3):381--392, Sep 1983.
\newblock \href {https://doi.org/10.1007/BF02579194}
  {\path{doi:10.1007/BF02579194}}.

\end{thebibliography}

\end{document}